\documentclass[12pt]{amsart}
\usepackage[margin=1in]{geometry}
\usepackage{latexsym}
\usepackage{float}
\usepackage{amssymb}
\usepackage[all]{xy}
\usepackage{epsfig}
\usepackage{color}
\usepackage{graphics}
\usepackage{hyperref}

\newtheorem{theorem}{Theorem}[section]

\newtheorem{lemma}[theorem]{Lemma}
\newtheorem{proposition}[theorem]{Proposition}
\newtheorem{corollary}[theorem]{Corollary}

\begin{document}
\begin{abstract}
In this paper,  a method is given to calculate the Jones polynomial of the 6-plat presentations of knots by using  a representation of the braid group $\mathbb{B}_6$ into a group of $5\times 5$ matrices.   We also can  calculate the Jones polynomial of the $2n$-plat presentations of knots by generalizing the method for the $6$-plat presentations of knots. 

\end{abstract}

\title{On the Jones Polynomial of  $2n$-plat presentations of knots}
\author{ Bo-hyun Kwon}
\date{Thursday, September 12, 2013}
\maketitle
\footnote{The subject classification code: 57M27}
\section{Introduction }

In  1985, Jones [8] discovered the polynomial knot invariant $V_K(t)$ and gave a formula to calculate the polynomials of knots that are presented as closed braids. He also gave a   formula to calculate the Jones polynomials of knots that are presented as closed plats. The closed plat formula is  described in [3]. Birman and Kanenobu [3] generalized the formula to the polynomials of knots which are obtained by a combination of closed braid and plat. In the case of $2n$-plat, by using the skein relation of Jones polynomial, we have $2^n$ closed braids that is related to the given $2n$-plat. Then the Jones polynomial of the $2n$-plat can be obtained from the Jones polynomials of the $2^n$ closed braids.\\

Kauffman [9] introduced the Kauffman bracket $<K>$ and the writhe $w(K)$ to calculate the Kauffman polynomial $X_K(a)$, which is identical to the Jones polynomial $V_K(t)$ with the change of variable $t=a^4$.\\

In this paper, by using the skein relation of the Kauffman bracket, we present a method  to calculate the Kauffman bracket and the writhe of 6-plat presentaions of knots that is obtained directly from the $6$-plat presentation. Also, we indicate how it extends to $2n$-plat presentations of knots.\\

Let $S^2$ be a sphere smoothly embedded in $S^3$ and let $K$ be a knot transverse to $S^2$. The complement in $S^3$ of $S^2$ consists of two open balls, $B_1$ and $B_2$. We assume that $S^2$ is the $xz$-plane $\cup~\{\infty\}$. Let $p$ be the projection onto the $xy$-plane  from $\mathbb{R}^3$.
Then the projection of $S^2-\{\infty\}$ onto the $xy$-plane  is the $x$-axis, and $B_1$ projects to the upper half plane.  Similarly, $B_2$ projects to the lower half plane.
The resulting diagram of  $K$  is called a plat on $2n$-strings, denoted by $p_{2n}(w)$,  if it is the union of a $2n$-braid $w$ and $2n$ unlinked and unknotted arcs which connect pairs of consecutive strings of the braid at the top and at the bottom endpoints and $S^2$ meets the top of the $2n$-braid. The bridge (plat) number $b(K)$ of $K$ is the smallest possible number $n$ such that there exists a plat presentation  of $K$ on $2n$ strings. We remark that the braid group $\mathbb{B}_{2n}$ is generated by $\sigma_1,\sigma_2,\cdot\cdot\cdot\sigma_{2n-1}$ which are twistings of two adjacent strings. For example, $w=\sigma_2^{-1}\sigma_4^{-1}\sigma_3\sigma_1^3\sigma_5^2\sigma_4^{-1}\sigma_2^{-1}$ is the word for the
6 braid in the dotted rectangle of the first diagram of Figure 1.\\
\begin{figure}[htb]
\includegraphics[scale=.30]{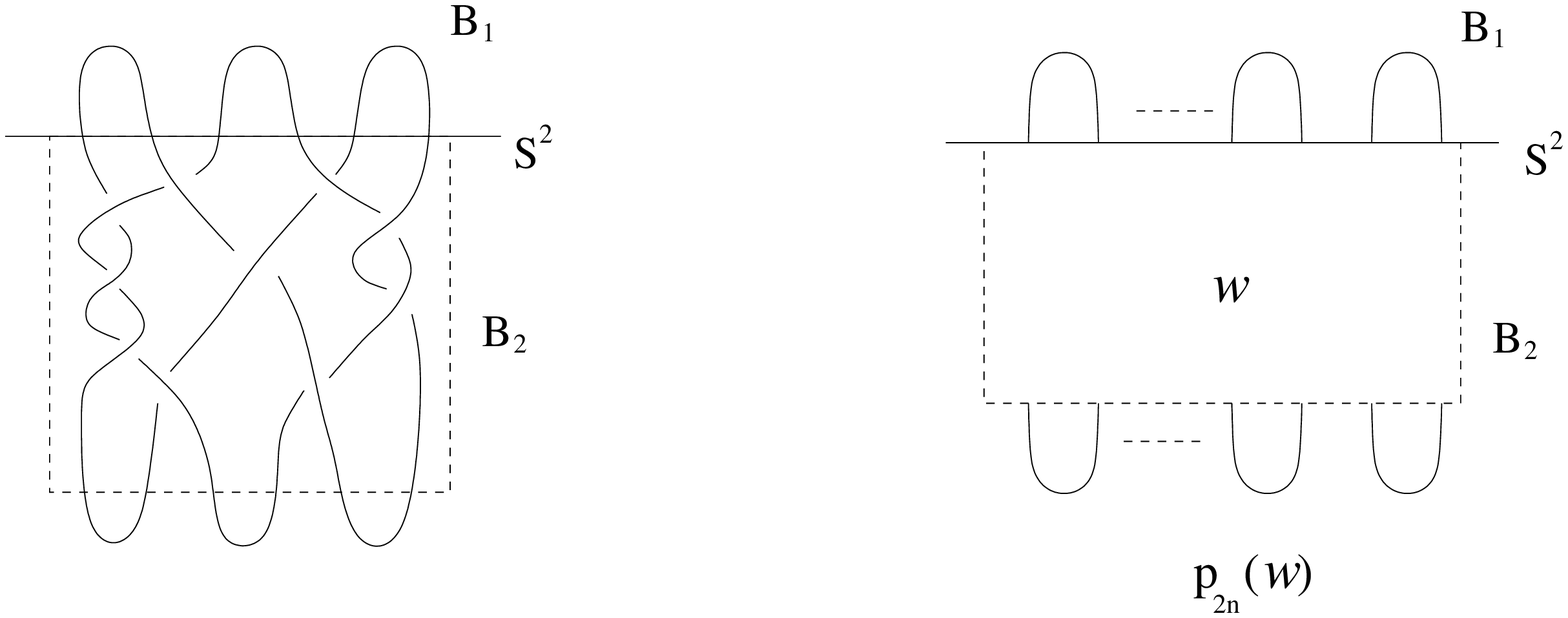}
\caption{}
\end{figure}\\

 Then we say that a plat presentation  is $\emph{standard}$ if the $2n$-braid $w$ of $p_{2n}(w)$  involves only $\sigma_2,\sigma_3,\cdot\cdot\cdot,\sigma_{2n-1}$.\\
 
Let $q_{2n}(w)=p_{2n}(w)\cap B_2$ be the $\emph{plat presentation}$  for the rational $n$-tangles $K\cap B_2$. [See [6].]\\

We say that $\overline{q_{2n}(w)}= p_{2n}(w)$  is the $\emph{plat closure}$ of $q_{2n}(w)$.\\

The tangle diagrams with the circles in Figure 3 give the  diagrams of trivial rational $2,3$-tangles as in $[1],[4],[7],[10]$. The right side of the diagrams show the trivial rational $2,3$-tangles in $B_2$.\\

We note that $q_{2n}(w)$  is alternating if and only if $\overline{q_{2n}(w)}$  is alternating.\\

A tangle $T$ is  $\emph{reduced}$ alternating if $T$ is alternating and $T$ does not have a self-crossing which can be removed by a Type I Reidemeister move. (See [1].)
We say that a knot $K$ is in \emph{n-bridge position} if the projection of $K$ onto the $xy$-plane has a plat presentation $p_{2n}(w)$. \\
\begin{figure}[htb]
\includegraphics[scale=.25]{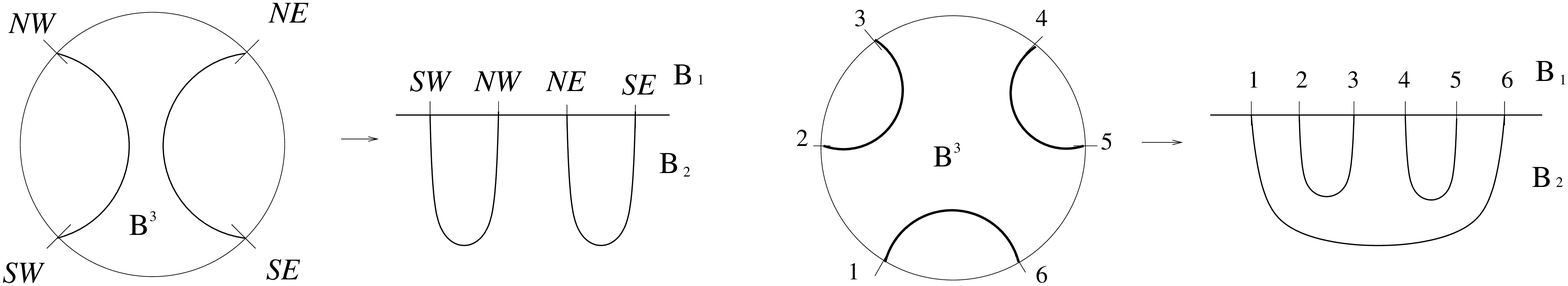}
\caption{}
\end{figure}

Let $\Lambda=\mathbb{Z}[a,a^{-1}]$ and $L$ be a link.\\

We recall that the Kauffman bracket $<L>\in \Lambda$ of a link $L$ is obtained from the following three axioms (See [1].)

\begin{figure}[htb]
\includegraphics[scale=.4]{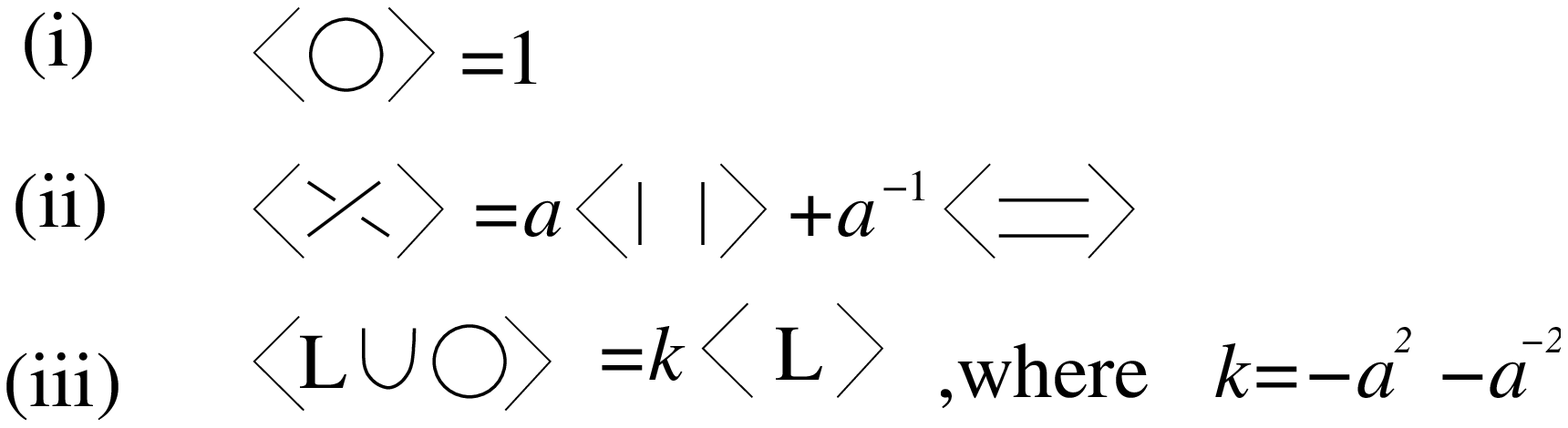}
\end{figure}

The symbol $< ~>$ indicates that the changes are made to the diagram locally, while the rest of the diagram is fixed.\\

 The Kauffman polynomial $X_L(a)\in \Lambda$ is defined by 
\begin{center}
$X_L(a)=(-a^{-3})^{w(\overrightarrow{L})}<L>$,
\end{center}
where the writhe $w(\overrightarrow{L})\in\mathbb{Z}$ is obtained by assigning an orientation to $L$, and taking a sum over all crossings of $L$ of their indices $e$, which are given by the following rule
 
\begin{figure}[htb]
 \includegraphics[scale=.4]{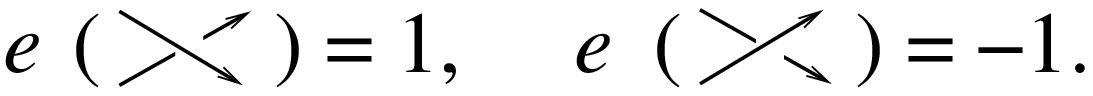}
\end{figure} 

In section 2, we introduce a theorem that explains how to calculate the Kauffman brackets for 4-plat presentations of knots.\\

In section 3, we show the main theorem that gives us a formula to calculate the Kauffman brackets for 6-plat presentations of knots.\\

In section 4, we generalize the formulas given in sections 3 to the Kauffman brackets of $2n$-plat presentations of knots.\\ 

Then, we give a method to calculate the writhes of of $n$-bridge presentations in section 5.\\

The author would like to thank his advisor Dr. Myers for his consistent encouragement and sharing his enlightening ideas on this topic.

\section{The Kauffman brackets of the 4-plat presentations of knots }

For given 2-tangles $T$ and $U$, we denote by $T+U$ the tangle sum of $T,~U$ and by $T*U$ the ``vertical sum" of $T,~U$ as in Figure 4.\\

\begin{figure}[htb]
\includegraphics[scale=.3]{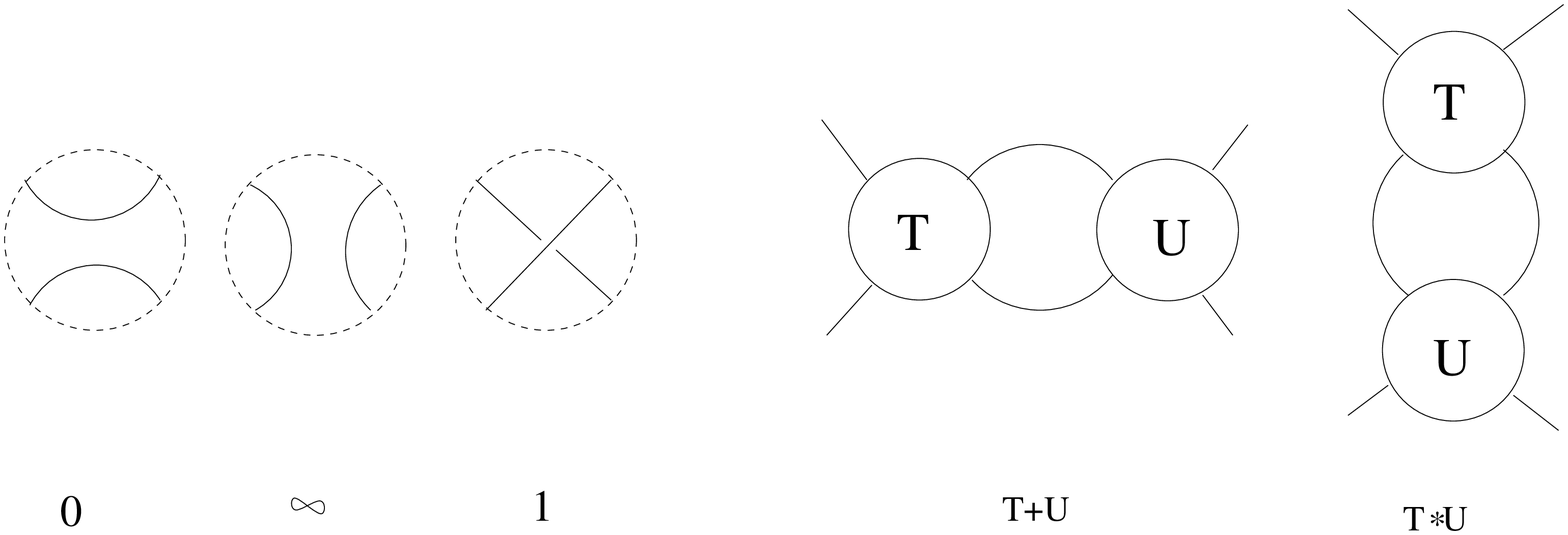}
\caption{the tangles 0,$\infty$, 1 and the tangle combinations $T+U,~T * U$.}
\end{figure}

Goldman and  Kauffman [6] define the $\emph{bracket polynomial}$ of the rational 2-tangle diagram $T$ as 
 $<T>=f^{T}(a)<0>+g^{T}(a)<{\infty}>$, where the coefficients $f^{T}(a)$ and $g^{T}(a)$ are Laurent polynomials  that are obtained by starting with $T$ and using the three axioms repeatedly until only the two trivial tangles $T_0$ and $T_{\infty}$ in the expression given for $T$ are left. Then we define the $\emph{bracket vector}$ of $T$ to be the ordered pair $[f^T(a),g^T(a)]^t$, and denote it by $br(T)$. For example, $br(1)=[a^{-1},a]^t$, where $1$ is the rational 2-tangle with only one positive crossing.\\
 
 Eliahou-Kauffman-Thistlethwaite [5] established the following.

\begin{proposition} For given  2-tangles $T$ and $U$, and $k=-a^2-a^{-2}$,\\

$br(T+U)=\left[\begin{array}{cc}
f^U(a) & 0\\
g^U(a) & f^U(a)+k g^U(a)\\
\end{array}
\right]br(T)$ and,
$ br(T\ast U)=\left[\begin{array}{cc}
k f^U(a)+g^U(a) & f^U(a)\\
0 &  g^U(a)\\
\end{array}
\right]br(T).$

 \end{proposition}
 
 So, if $U=1$ in Proposition 2.1 then we have the following equalities.\\

$br(T+1)=M_+\cdot br(T)$, $br(T*1)=M_*\cdot br(T)$, where \\

$M_+= \left[ \begin{array}{cc}
a^{-1} & 0 \\
a & -a^{3} \\
 \end{array} \right]$ and $M_*= \left[ \begin{array}{cc}
-a^{-3} & a^{-1} \\
0 & a \\
 \end{array} \right].$\\

 Two rational $2$-tangles, $T,T'$, in $B^3$ are $\emph{isotopic}$, denoted by $T\sim T'$, if there is an orientation-preserving self-homeomorphism
$h: (B^3, T)\rightarrow (B^3,T')$ that is the identity map on the boundary. \\

We say that $T^{hflip}$ is the $\emph{horizontal flip}$ of the 2-tangle $T$ if $T^{hflip}$ is obtained from $T$ by a $180^\circ$-rotation around a horizontal axis on the plane of $T$, and $T^{vflip}$ is the $\emph{vertical flip}$ of the tangle $T$ if $T^{vflip}$ is obtained from $T$ by a $180^\circ$-rotation around a vertical axis, see Figure 5 for illustrations. Then we have the following lemma by Kauffman.
 
\begin{lemma}$($Flipping Lemma $[7])$
If $T$ is rational 2-tangle, then $T\sim T^{vflip}$ and $T\sim T^{hflip}$.
\end{lemma} 

We note that any rational 2-tangle $T$ can be obtained from an element $u$ of the braid group $\mathbb{B}_3$ as in the first bottom diagram of Figure 5. (Refer to [7].)\\

For $u=\sigma_{k_1}^{\epsilon_1}\sigma_{k_2}^{\epsilon_2}\cdot\cdot\cdot\sigma_{k_n}^{\epsilon_n}$, the reverse word of $u$, denoted by $u^r$, is defined by the word $u^r=\sigma_{k_n}^{\epsilon_n}\sigma_{k_{n-1}}^{\epsilon_{n-1}}\cdot\cdot\cdot \sigma_{k_1}^{\epsilon_1}$.
Then by Lemma 2.2, we see  how to get a word $u^r$ for a 4-plat presentation of a rational 2-tangle $T$ as in the bottom diagrams of Figure 5.\\

 \begin{figure}[htb]
\includegraphics[scale=.25]{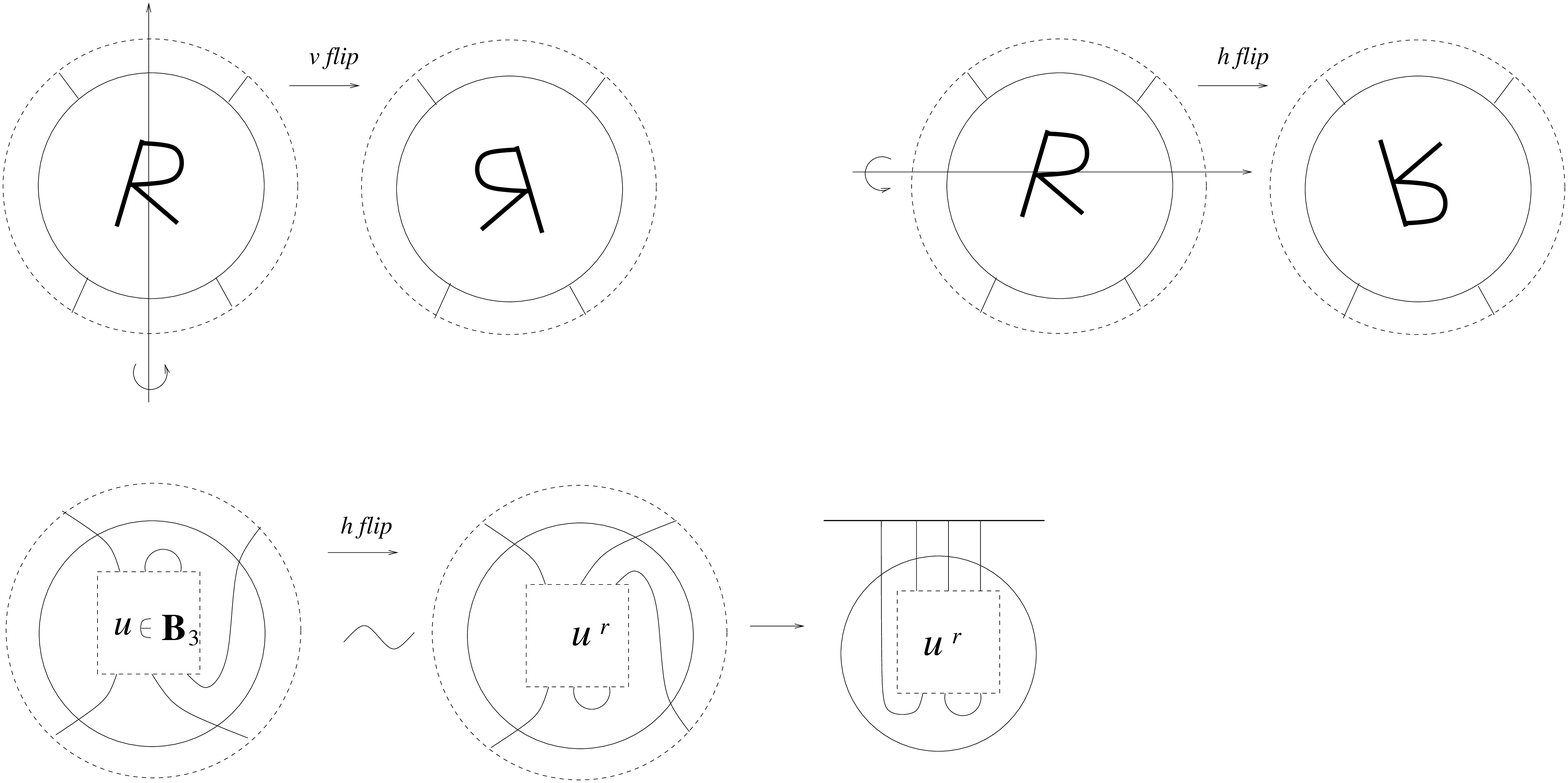}
\caption{}
\end{figure}
 
 Suppose that $R$ is a rational 2-tangle. Let $u$ be the word for a standard 4-plat presentation of $R$.
Then, by modifying the diagrams of $R+1$ and $R\ast 1$ as in Figure 6,  we see that $u'=\sigma_3^{-1}w$ and $u''=\sigma_2w$, where $u'$ and $u''$ are the words for 4-plat presentations of $R+1$ and $R\ast 1$ respectively.\\

\begin{figure}[htb]
\includegraphics[scale=.3]{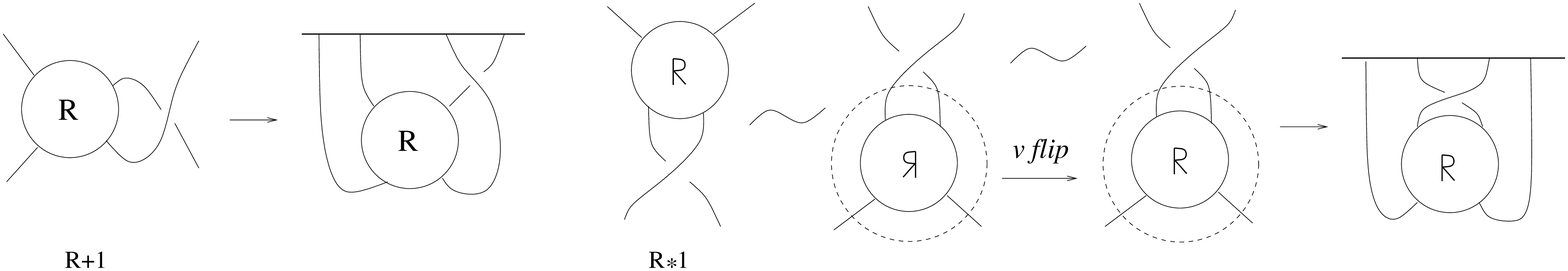}
\caption{}
\end{figure}

Now, consider the following theorem.
 
 \begin{theorem}$[$Conway$,($See$ ~[10])]$
If $K$ is a 2-bridge knot, then there exists a word $w$ in $\mathbb{B}_4$ so that the  plat presentation $p_4(w)$ is reduced alternating and standard and represents a knot isotopic to $K$.
\end{theorem}

 Let $A_2=M_*$ and  $A_3=M_+^{-1}$.\\
 
 Then we can derive the following theorem which shows how to calculate the Jones polynomials of 4-plat presentations of knots.
 
\begin{theorem}
Suppose that $q_4(w)$ is a  plat presentation of a rational 2-tangle $T$ which is reduced alternating and standard so that
$w=\sigma_3^{-\epsilon_1}\sigma_2^{\epsilon_2}\cdot\cdot\cdot\sigma_3^{-\epsilon_{2n-1}}\sigma_2^{\epsilon_{2n}}$ for some positive  integers $\epsilon_i$ ($2\leq i\leq 2n$) and non-negative integer $\epsilon_1$. Let $A=A_3^{-\epsilon_1}A_2^{\epsilon_2}\cdot\cdot\cdot A_3^{-\epsilon_{2n-1}}A_2^{\epsilon_{2n}}$. Then,\\

\hskip 10pt $<T>=f^T(a)<T_0>+g^T(a)<T_{\infty}>$, where $f^T(a)$ and $g^T(a)$ are given by
$br(T)=[f^T(a),g^T(a)]^t=A[0,1]^t$ and $<K>=f^T(a)-(a^2+a^{-2})g^T(a)$ for $K$, where $K$ is represented by the  plat presentation $\overline{q_4(w)}$.\\

Therefore, $X_K=(-a^{-3})^{w({K})}(f^T(a)-(a^2+a^{-2})g^T(a))$, where $w(K)$ is the writhe of the knot $K$.
\end{theorem}

\begin{proof}

Let $l=|\epsilon_1|+|\epsilon_2|+\cdot\cdot\cdot+|\epsilon_{2n}|$.\\

We will show this theorem by using induction on $l$.\\

Suppose that $w=\sigma_2$. Then $T=\infty\ast 1$.\\

Therefore, by Proposition 2.1., $br(T)=M_{*}[0,1]^t=A_2[0,1]^t$.\\

Now, we assume that $br(T)=[f^T(a),g^T(a)]^t=A[0,1]^t$ if $T$ is a reduced alternating standard rational 2-tangle with the plat presentation $q_4(w)$, where $w=\sigma_3^{-\epsilon_1}\sigma_2^{\epsilon_2}\cdot\cdot\cdot\sigma_2^{\epsilon_{2n}}$ for some positive  integers $\epsilon_i$ ($2\leq i\leq 2n$) and non-negative integer $\epsilon_1$, and $l=|\epsilon_1|+|\epsilon_2|+\cdot\cdot\cdot+|\epsilon_{2n}|=k$.\\

Now, consider a reduced alternating standard rationl 2-tangle $T'$ with the plat presentation $q_4(w')$, where  $w'=\sigma_3^{-\epsilon_1'}\sigma_2^{\epsilon_2'}\cdot\cdot\cdot\sigma_2^{\epsilon_{2m}'}$ for some positive  integers $\epsilon_i$ ($2\leq i\leq 2m$) and non-negative integer $\epsilon_1'$, and $l=|\epsilon_1'|+|\epsilon_2'|+\cdot\cdot\cdot+|\epsilon_{2m}'|=k+1$.\\

If $\epsilon_i'\geq 1$ then we set $w''=\sigma_3^{-(\epsilon_1'-1)}\sigma_2^{\epsilon_2'}\cdot\cdot\cdot\sigma_2^{\epsilon_{2m}'}$. Then $w'=\sigma_3^{-1}w''$.\\

Let $T''$ be the reduced alternating standard rational 2-tangle with the plat presentation $q_4(w'')$.\\

Let $A'=A_3^{-\epsilon_1'}A_2^{\epsilon_2'}\cdot\cdot\cdot A_2^{\epsilon_{2m}'}$ and $A''=A_3^{-(\epsilon_1'-1)}A_2^{\epsilon_2'-1}\cdot\cdot\cdot A_2^{\epsilon_{2m}'}$.\\

Since $|\epsilon_1'-1|+|\epsilon_2'|+\cdot\cdot\cdot+|\epsilon_{2m}'|=k$, we note that $br(T'')=[f^{T''}(a),g^{T''}(a)]^t=A''[0,1]^t$ by assumption.
We note that $T'=T''+1$. So, by Proposition 2.1., $br(T')=M_+\cdot br(T'')=A_3^{-1}(A''[0,1]^t)=(A_3^{-1}A'')[0,1]^t=A'[0,1]^t$.\\

If $\epsilon_1'=0$ then we set $w''=\sigma_2^{\epsilon_2'-1}\cdot\cdot\cdot\sigma_2^{\epsilon_{2m}'}$. Then $w'=\sigma_2w''$.\\

Let $T''$ be the reduced alternating standard rational 2-tangle with the plat presentation $q_4(w'')$.\\
Let $A'=A_2^{\epsilon_2'}\cdot\cdot\cdot A_2^{\epsilon_{2m}'}$ and $A''=A_2^{\epsilon_2'-1}\cdot\cdot\cdot A_2^{\epsilon_{2m}'}$.\\

Since $|\epsilon_2'-1|+\cdot\cdot\cdot+|\epsilon_{2m}'|=k$, we note that $br(T'')=[f^{T''}(a),g^{T''}(a)]^t=A''[0,1]^t$ by assumption.\\
We note that $T'=T''\ast 1$. So, by Proposition 2.1., $br(T')=M_*\cdot br(T'')=A_2(A''[0,1]^t)=(A_2A'')[0,1]^t=A'[0,1]^t$.
\end{proof}

Now, assume that $q_4(w)$ be a 4-plat presentation of a rational 2-tangle $T$ which is reduced alternating and standard so that $w=\sigma_3^{-\epsilon_1}\sigma_2^{\epsilon_2}\cdot\cdot\cdot\sigma_3^{-\epsilon_{2n-1}}\sigma_2^{\epsilon_n}$ for some negative integers $\epsilon_i$ ($2\leq i\leq 2n$) and non-positive integer $\epsilon_1$.\\

We can calculate the Kauffman bracket of $q_4((w^{-1})^{r})$ by usinge the previous theorem.\\

We note that $q_4(w)$ is the mirror image of  $q_4((w^{-1})^{r})$ which is obtained by interchanging the over and under crossings.\\

So, we switch $a$ and $a^{-1}$ to calculate the Kauffman bracket of the 4-plat presentation $q_4(w)$ of the rational 2-tangle $T$.

\section{The Kauffman brackets of the 6-plat presentations of knots}

Now, suppose that $K$ is in  3-bridge position. Then we have a plat presentation $q_6(w)$ for the rational 3-tangle $K\cap B_2$.
Then, let $w=\sigma_{k_1}^{\epsilon_1}\sigma_{k_2}^{\epsilon_2}\cdot\cdot\cdot
\sigma_{k_{n-1}}^{\epsilon_{n-1}}\sigma_{k_{n}}^{\epsilon_{n}}$ for some non-zero integers $\epsilon_i$ ($1\leq i\leq n$), where $k_i\in\{1,2,3,4,5\}$.\\

\begin{figure}[htb]
\includegraphics[scale=.35]{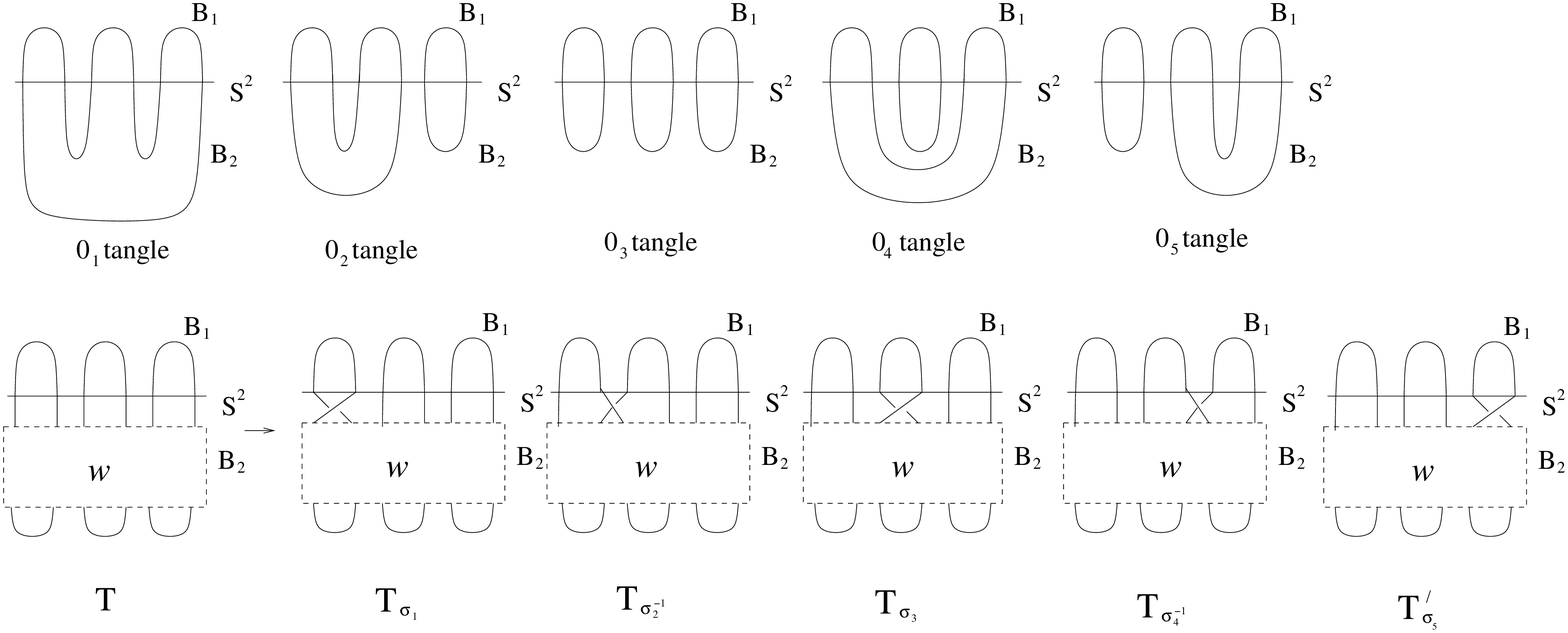}
\caption{}
\end{figure}

H. Cabrera-Ibarra [4] defined the bracket polynomial of the rational 3-tangle $T$ as $<T>=f^T_1(a)<{0_1}>+f^T_2(a)<{0_2}>+f^T_3(a)<{0_3}>+f^T_4(a)<{0_4}>+f^T_5<{0_5}>$, where 
$f^T_i(a)$ are polynomials in $a$ and $a^{-1}$ that are obtained by starting with $T$ and using the three axioms repeatedly until only the five trivial tangles $<{0_j}>$ in the expression given for $T$ are left. (See Figure 6 and 7.)   \\

Let $\mathcal{A}=<{0_1}>, \mathcal{B}=<{0_2}>, \mathcal{C}=<{0_3}>, \mathcal{D}=<{0_4}>$ and $\mathcal{E}=<{0_5}>$.\\

Let $B_1^{\pm 1}= \left[ \begin{array}{ccccc}
a^{\pm 1} &0& 0& 0& 0 \\
0 & a^{\pm 1} &0 &0 &0 \\
0& a^{\mp 1}& -a^{\mp 3}& a^{\mp 1}&0 \\
0& 0&0 &a^{\pm 1} &0\\
a^{\mp 1}& 0& 0& 0&-a^{\mp 3}\\
 \end{array} \right],$
\hskip 18pt $B_2^{\pm 1}= \left[ \begin{array}{ccccc}
-a^{\mp 3}&0& 0& a^{\mp 1}& a^{\mp 1} \\
0 & -a^{\mp 3} &a^{\mp 1} &0 &0 \\
0& 0& a^{\pm 1}& 0&0 \\
0& 0&0 &a^{\pm 1} &0\\
0& 0& 0& 0&a^{\pm 1}\\
 \end{array} \right]$,\\
\vskip 15pt

$~\hskip 18pt B_3^{\pm 1}= \left[ \begin{array}{ccccc}
a^{\pm 1}&0& 0& 0& 0 \\
0 & a^{\pm 1} &0 &0 &0 \\
0& a^{\mp 1}& -a^{\mp 3}& 0&a^{\mp 1} \\
a^{\mp 1}& 0&0 &-a^{\mp 3} &0\\
0& 0& 0& 0&a^{\pm 1}\\
 \end{array} \right]$,~\hskip 18pt$B_4^{\pm 1}= \left[ \begin{array}{ccccc}
-a^{\mp 3}& a^{\mp 1}& 0&a^{\mp 1}&0 \\
0 & a^{\pm 1} &0 &0 &0 \\
0& 0& a^{\pm 1}& 0&0 \\
0& 0&0 &a^{\pm 1} &0\\
0& 0& a^{\mp 1}& 0&-a^{\mp 3}\\
 \end{array} \right],\\$
 \vskip 15pt

$~\hskip 18ptB_5^{\pm 1}=\left[ \begin{array}{ccccc}
a^{\pm 1}&0& 0& 0& 0 \\
a^{\mp 1}& -a^{\mp 3} &0 &0 &0 \\
0& 0& -a^{\mp 3}& a^{\mp 1}&a^{\mp 1} \\
0& 0&0 &a^{\pm 1} &0\\
0& 0& 0& 0&a^{\pm 1}\\
 \end{array} \right]. $\\
 
Let $B=B_{k_1}^{\epsilon_1}B_{k_2}^{\epsilon_2}\cdot\cdot\cdot B_{k_{n-1}}^{\epsilon_{n-1}}B_{k_{n}}^{\epsilon_{n}}$.\\

Then we have the following theorem to calculate the Kauffman polynomial of $K$. 
\begin{theorem} Suppose that $q_6(w)$ is a  plat presentation for a rational 3-tangle $T$ and $w=\sigma_{k_1}^{\epsilon_1}\sigma_{k_2}^{\epsilon_2}\cdot\cdot\cdot
\sigma_{k_{n-1}}^{\epsilon_{n-1}}\sigma_{k_{n}}^{\epsilon_{n}}$ for some non-zero integers $\epsilon_i$ ($1\leq i\leq n$), where $k_i\in\{1,2,3,4,5\}$.\\

 Then $ <T>=f^T_1(a)\mathcal{A}+f^T_2(a)\mathcal{B}+f^T_3(a)\mathcal{C}+f^T_4(a)\mathcal{D}+f^T_5(a)\mathcal{E}$ , where $f^T_i(a)$ are given by\\

 $[f^T_1(a)~ f^T_2(a)~f^T_3(a)~f^T_4(a)~f^T_5(a)]^t=B[0 ~0~1~0~0]^r,$  and $B= B_{k_1}^{\epsilon_1}B_{k_2}^{\epsilon_2}\cdot\cdot\cdot
B_{k_{n-1}}^{\epsilon_{n-1}}B_{k_{n}}^{\epsilon_{n}}$. (i.e., the third column of $B$)\\

Also, $<K>=f^T_1(a)+k(f^T_2(a)+f^T_4(a)+f^T_5(a))+k^2f^T_3(a)$, where $k=-a^2-a^{-2}$ and $K$ is the knot which is represented by the  plat presentation $\overline{q_6(w)}$.\\

Therefore, $X_K=(-a^{-3})^{w(\overrightarrow{K})}(f^T_1(a)+k(f^T_2(a)+f^T_4(a)+f^T_5(a))+k^2f^T_3(a))$.
\end{theorem}

\begin{proof}
Suppose that $K$ is a 3-bridge link. Then, we have a link $K'$ which is isotopic to $K$ and the projection onto the $xy$-plane has a  plat presentation $p_6(w)$. Then we define $q_6(w)$ that is the  plat presentation of the tangle $T=K'\cap B_2$ as in the first bottom diagram of Figure 6.\\

Suppose that $<T>=f^T_1(a)\mathcal{A}+f^T_2(a)\mathcal{B}+f^T_3(a)\mathcal{C}+f^T_4(a)\mathcal{D}+f^T_5(a)\mathcal{E}$ for some polynomials $f^T_i(a)$.\\

Let $T_{\sigma_j^{\pm 1}}$ be the new rational 3-tangle in $B_2$ which is obtained from $T$ by adding $\sigma_j^{\pm 1}$ for $1\leq j\leq 5$ as in the bottom diagrams of Figure 6.\\

Suppose that $<T_{{\sigma_j^{\pm 1}}}>=f^{T_{{\sigma_j^{\pm 1}}}}_1(a)\mathcal{A}+f^{T_{{\sigma_j^{\pm 1}}}}_2(a)\mathcal{B}+f^{T_{{\sigma_j^{\pm 1}}}}_3(a)\mathcal{C}+f^{T_{{\sigma_j^{\pm 1}}}}_4(a)\mathcal{D}+f^{T_{{\sigma_j^{\pm 1}}}}_5(a)\mathcal{E}$ for some polynomials $f^{T_{{\sigma_j^{\pm 1}}}}_i(a)$.\\

For convenience, let $f_i(a)=f_i^T(a)$ and $f'_i(a)=f^{T_{{\sigma_j^{\pm 1}}}}_i(a)$.\\

Then,  we get $<T_{\sigma_1}>=a<T>+a^{-1}<T'>$. We note that $<T'>=f_1(a)\mathcal{E}+f_2(a)\mathcal{C}+kf_3(a)\mathcal{C}+f_4(a)\mathcal{C}+kf_5(a)\mathcal{E}$ as in Figure 7, where $k=-(a^2+a^{-2}).$\\

\begin{figure}[htb]
\includegraphics[scale=.28]{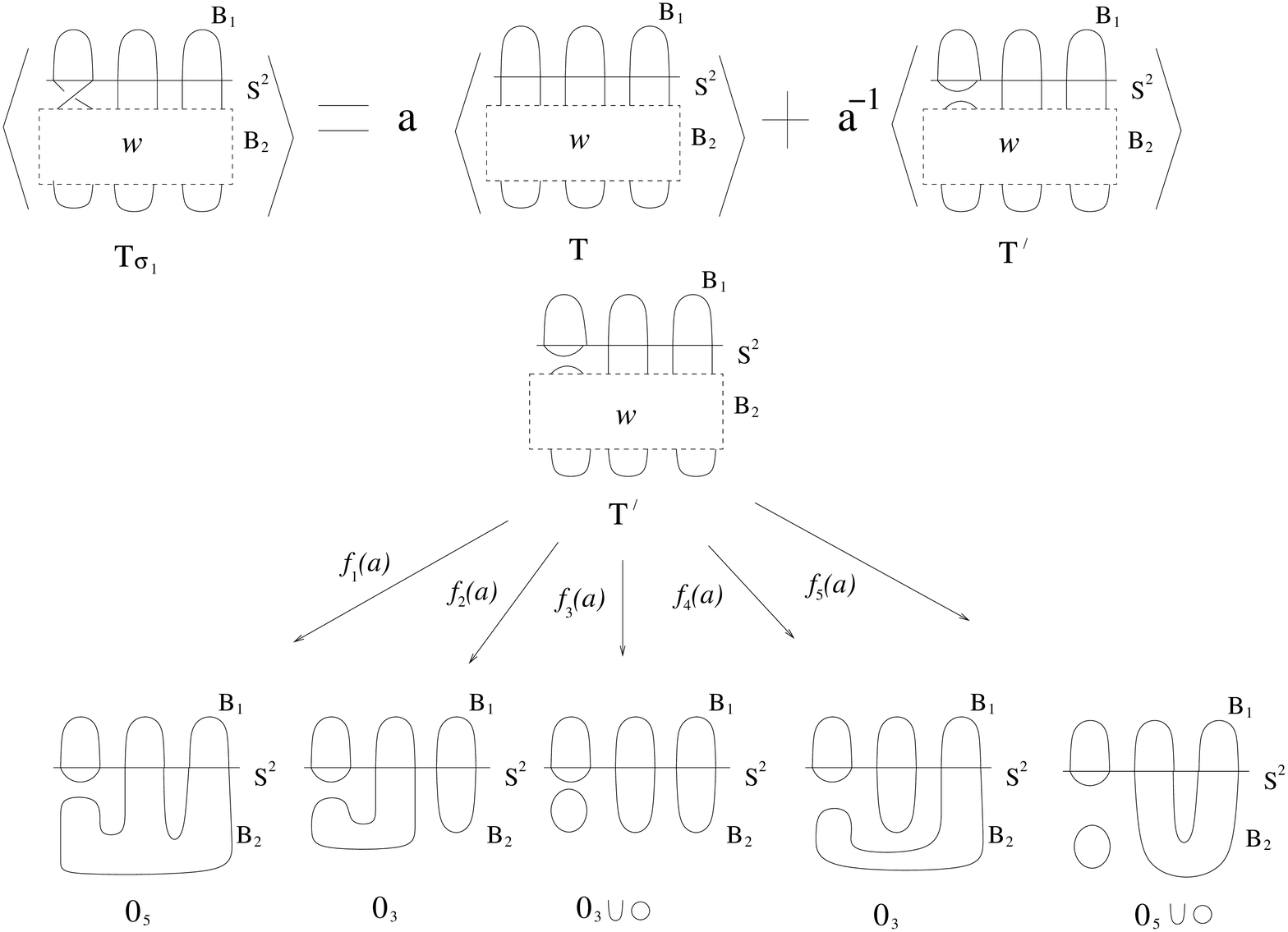}
\caption{}
\end{figure}

Therefore, $
<T_{\sigma_1}>=a<T>+a^{-1}<T'>=
 a(f_1(a)\mathcal{A}+f_2(a)\mathcal{B}+f_3(a)\mathcal{C}+f_4(a)\mathcal{D}+f_5(a)\mathcal{E})+a^{-1}(f_1(a)\mathcal{E}+f_2(a)\mathcal{C}+kf_3(a)\mathcal{C}+f_4(a)\mathcal{C}+kf_5(a)\mathcal{E})=
  af_1(a)\mathcal{A}+af_2(a)\mathcal{B}+(af_3(a)+a^{-1}(f_2(a)+kf_3(a)+f_4(a)))\mathcal{C}+af_4(a)\mathcal{D}+(af_5(a)+a^{-1}(f_1(a)+kf_5(a)))\mathcal{E}.$\\
  
  So, we have $f'_1(a)=af_1(a)$, $f'_2(a)=af_2(a)$, $f'_3(a)=af_3(a)+a^{-1}(f_2(a)+kf_3(a)+f_4(a))$, $f'_4(a)=af_4(a)$ and $f'_5(a)=af_5(a)+a^{-1}(f_1(a)+kf_5(a))$.

\begin{figure}[htb]
\includegraphics[scale=.28]{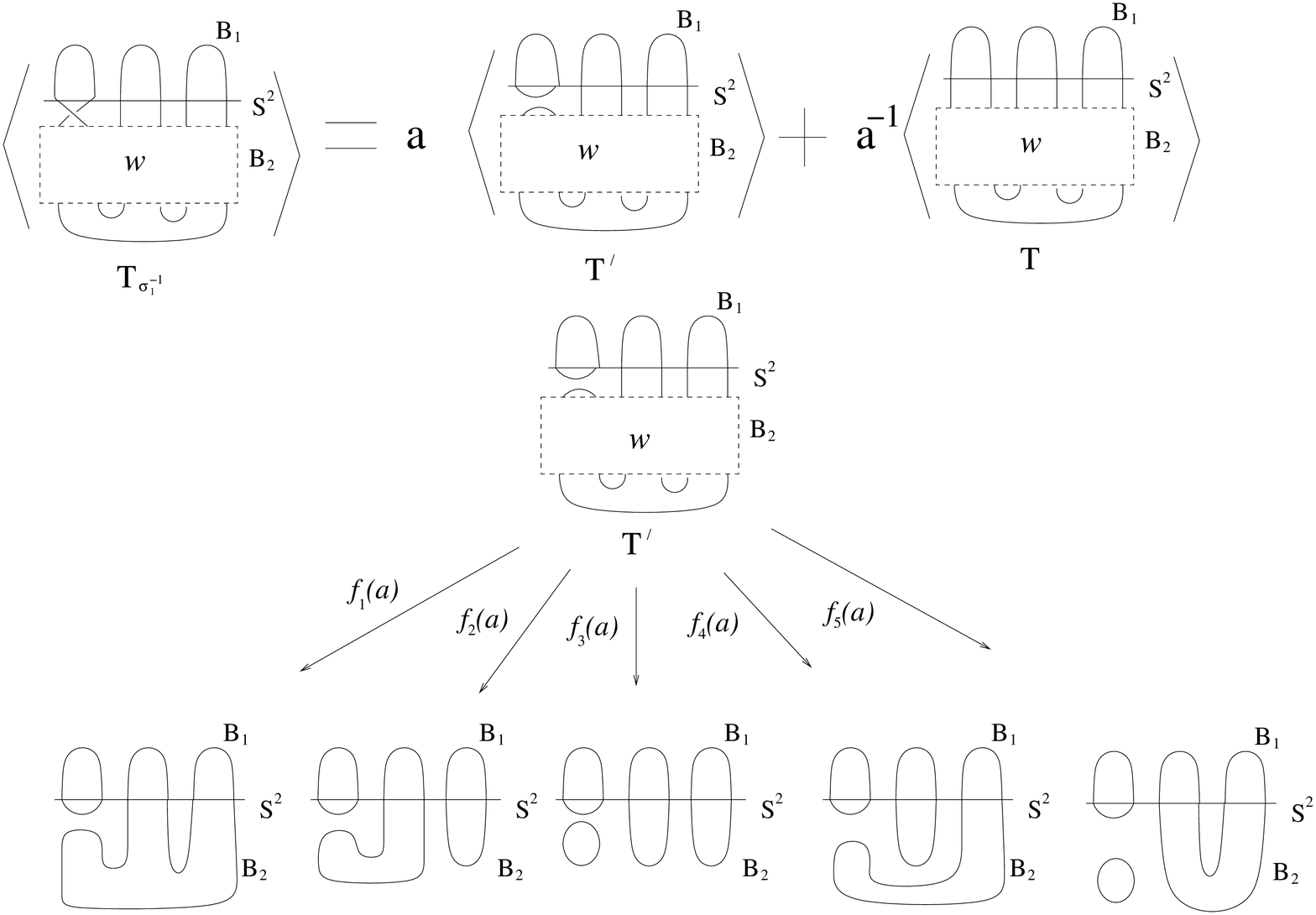}
\caption{}
\end{figure}

Similarly, By Figure 8, we have $<T_{\sigma_1^{-1}}>=a<T'>+a^{-1}<T>$.\\

Therefore, $<T_{\sigma_1^{-1}}>=a<T'>+a^{-1}<T>=
 a(f_1(a)\mathcal{E}+f_2(a)\mathcal{C}+kf_3(a)\mathcal{C}+f_4(a)\mathcal{C}+kf_5(a)\mathcal{E})+a^{-1}(f_1(a)\mathcal{A}+f_2(a)\mathcal{B}+f_3(a)\mathcal{C}+f_4(a)\mathcal{D}+f_5(a)\mathcal{E})=
  a^{-1}f_1(a)\mathcal{A}+a^{-1}f_2(a)\mathcal{B}+(a^{-1}f_3(a)+a(f_2(a)+kf_3(a)+f_4(a)))\mathcal{C}+a^{-1}f_4(a)\mathcal{D}+(a^{-1}f_5(a)+a(f_1(a)+kf_5(a)))\mathcal{E}.$\\

So, we have $f'_1(a)=a^{-1}f_1(a)$, $f'_2(a)=a^{-1}f_2(a)$, $f'_3(a)=a^{-1}f_3(a)+a(f_2(a)+kf_3(a)+f_4(a))$, $f'_4(a)=a^{-1}f_4(a)$ and $f'_5(a)=a^{-1}f_5(a)+a(f_1(a)+kf_5(a))$.\\

This operations can be expressed by the following.
\vskip 10pt

$ \left[ \begin{array}{ccccc}
a^{\pm 1} &0& 0& 0& 0 \\
0 & a^{\pm 1} &0 &0 &0 \\
0& a^{\mp 1}& a^{\pm 1}+a^{\mp 1}k& a^{\mp 1}&0 \\
0& 0&0 &a^{\pm 1} &0\\
a^{\mp 1}& 0& 0& 0&a^{\pm 1}+a^{\mp 1}k\\
 \end{array} \right]$
$ \left[ \begin{array}{c}
f_1(a)\\ f_2(a)\\f_3(a)\\f_4(a)\\f_5(a)\\

 \end{array} \right]=\left[ \begin{array}{c}
f_1'(a) \\ f_2'(a) \\f_3'(a) \\f_4'(a) \\f_5'(a) \\
\end{array} \right]$\\

 Similarly, we have four more operations from $\sigma_2^{\pm 1}, \sigma_3^{\pm 1}$, $\sigma_4^{\pm 1}$ and $\sigma_5^{\pm 1}$ as follows.
 
 \vskip 10pt
 $ \left[ \begin{array}{ccccc}
a^{\pm 1} +a^{\mp 1}k&0& 0& a^{\mp 1}& a^{\mp 1} \\
0 & a^{\pm 1}+a^{\mp 1}k &a^{\mp 1} &0 &0 \\
0& 0& a^{\pm 1}& 0&0 \\
0& 0&0 &a^{\pm 1} &0\\
0& 0& 0& 0&a^{\pm 1}\\
 \end{array} \right]\left[ \begin{array}{c}
f_1(a) \\ f_2(a) \\f_3(a) \\f_4(a) \\f_5(a)\\

 \end{array} \right]= \left[ \begin{array}{c}
f_1'(a) \\ f_2'(a) \\f_3'(a) \\f_4'(a) \\f_5'(a) \\

 \end{array} \right]$
 \vskip 10pt

$ \left[ \begin{array}{ccccc}
a^{\pm 1}&0& 0& 0& 0 \\
0 & a^{\pm 1} &0 &0 &0 \\
0& a^{\mp 1}& a^{\mp 1}k+a^{\pm 1}& 0&a^{\mp 1} \\
a^{\mp 1}& 0&0 &a^{\mp 1}k+a^{\pm 1} &0\\
0& 0& 0& 0&a^{\pm 1}\\
 \end{array} \right] \left[ \begin{array}{c}
f_1(a) \\ f_2(a) \\f_3(a) \\f_4(a) \\f_5(a)\\

 \end{array} \right]= \left[ \begin{array}{c}
f_1'(a) \\ f_2'(a) \\f_3'(a) \\f_4'(a) \\f_5'(a) \\

 \end{array} \right]$
 \vskip 20pt

$ \left[ \begin{array}{ccccc}
a^{\pm 1} +a^{\mp 1}k& a^{\mp 1}& 0&a^{\mp 1}&0 \\
0 & a^{\pm 1} &0 &0 &0 \\
0& 0& a^{\pm 1}& 0&0 \\
0& 0&0 &a^{\pm 1} &0\\
0& 0& a^{\mp 1}& 0&a^{\pm 1}+a^{\mp 1}k\\
 \end{array} \right] \left[ \begin{array}{c}
f_1(a) \\ f_2(a) \\f_3(a) \\f_4(a)\\f_5(a)\\

 \end{array} \right]= \left[ \begin{array}{c}
f_1'(a) \\ f_2'(a) \\f_3'(a) \\f_4'(a) \\f_5'(a) \\

 \end{array} \right]$
 \vskip 20pt

$ \left[ \begin{array}{ccccc}
a^{\pm 1}&0& 0& 0& 0 \\
a^{\mp 1}& a^{\pm 1}+a^{\mp 1}k &0 &0 &0 \\
0& 0& a^{\pm 1}+a^{\mp 1}k& a^{\mp 1}&a^{\mp 1} \\
0& 0&0 &a^{\pm 1} &0\\
0& 0& 0& 0&a^{\pm 1}\\
 \end{array} \right] \left[ \begin{array}{c}
f_1(a) \\ f_2(a) \\f_3(a) \\f_4(a) \\f_5(a)\\

 \end{array} \right]= \left[ \begin{array}{c}
f_1'(a) \\f_2'(a) \\f_3'(a) \\f_4'(a)\\f_5'(a) \\

 \end{array} \right]$, where \hskip 10pt $k=-a^2-a^{-2}.$
 \vskip 20pt
 
 Recall that $T$ is expressed by $\sigma_{k_1}^{\epsilon_1}\sigma_{k_2}^{\epsilon_2}\cdot\cdot\cdot
 \sigma_{k_{n-1}}^{\epsilon_{n-1}}\sigma_{k_{n}}^{\epsilon_{n}}$ for some non-zero integers $\epsilon_i$ ($1\leq i\leq n$), where $k_i\in\{1,2,3,4,5\}$.\\
 
 Then, we know that the generators $\sigma_i$ corresond to $B_i$. \\
 
 So, given $B=B_{k_1}^{\epsilon_1}B_{k_2}^{\epsilon_2}\cdot\cdot\cdot B_{k_{n-1}}^{\epsilon_{n-1}}B_{k_{n}}^{\epsilon_{n}}$ , we have\\
 
$ <T>=f_1(a)<T_1>\mathcal{A}+f_2(a)\mathcal{B}+f_3(a)\mathcal{C}+f_4(a)\mathcal{D}+f_5(a)\mathcal{E}$ \\ 

for $[f_1(a)~ f_2(a)~f_3(a)~f_4(a)~f_5(a)]^t=B[0~0~1~0~0]^t$ since $<T_{0_3}>=0\cdot\mathcal{A}+0\cdot\mathcal{B}+1\cdot\mathcal{C}+0\cdot\mathcal{D}+0\cdot\mathcal{E}$.\\

 From $<T>$, we have $<K>=f_1(a)+k(f_2(a)+f_4(a)+f_5(a))+k^2f_3(a)$ since $\overline{T_{0_1}}$ is the unknot, $\overline{T_{0_i}}$ for $i\in\{2,4,5\} $ are disjoint union of two unknots and $\overline{T_{0_3}}$ is disjoint union of three unknots as in Figure 8.\\

 Therefore, $X_K=(-a^{-3})^{w(\overrightarrow{K})}(f_1(a)+k(f_2(a)+f_4(a)+f_5(a))+k^2f_3(a))$
 \end{proof}

We remark that the matrices $B_1^{\pm 1},B_2^{\pm 1},B_3^{\pm 1}$, $B_4^{\pm 1}$ and $B_5^{\pm 1}$ satisify the braid group relations.

\section{The Kauffman brackets of $2n$-plat presentation knots }

We define the bracket polynomial of the rational $n$-tangle $T$ as $<T>=f^T_1(a)<{0_1}>+f^T_2(a)<{0_2}>+\cdot\cdot\cdot+f^T_m<{0_m}>$, where 
$f^T_i(a)$ are Laurent polynomials  that are obtained by starting with $T$ and using the three axioms repeatedly until only the $m$ trivial tangles $<{0_j}>$ in the expression given for $T$ are left. \\

So, we remark that the number $m$ of trivial rational $n$-tangles ${0_i}$  needs to be calculated. \\

To do this, let $\psi(0)=1$. Then  we define the map $\psi:2\mathbb{Z}^+\rightarrow \mathbb{Z}^+$ so that $\psi(2k)=\sum_{i=1}^{k}\psi(2i-2)\cdot \psi(2k-2i)$. \\

\begin{lemma}
The number of trivial rational $n$-tangles is $\psi(2n)$.
\end{lemma}

\begin{proof}
Let ${0_i}$ be a trivial rational $n$-tangle. Then we note that the string with the endpoint $1$ has the other endpoint at $2k+1$ for some positive integer $k$.\\

If $k=1$ then we calculate the number of trivial $n$-tangles by considering $2n-2$ endpoints and it is $\phi(0)\cdot\psi(2n-2)$.\\

If $k=2$ then we have a nested string inside of the string with the endpoint $1$ and we need to calculate the possible case for the rest of strings.
Then it is $\psi(2)\cdot\psi(2n-4)$.\\

By considering the all subcases with respect to $k$, we calculate the number of trivial rational $n$-tangle which is $\sum_{i=1}^{n}\psi(2i-2)\cdot \psi(2n-2i)$.\\

Therefore, the number of trivial rational $n$-tangles is $\psi(2n)$.

\end{proof}

Recall $\overline{T}$ that is the tangle closure of the tangle $T$ to have the knot with the $2n$-plat presentation.\\

Now, we have a corollary to calculate the Kauffman polynomial of $n$-plat presentation as follows.\\

\begin{corollary}
There exist $(4n-2)$ $\psi(2n)\times\psi(2n)$ matrices to calculate the coefficents $f^T_1(a),...,f^T_{\psi(2n)}(a)$ of the Kauffman bracket for a rational $n$-tangle $T$ with a $2n$-plat presentation ${q_{2n}(w)}$.
Moreover, we calculate the Kauffman bracket of $\overline{q_{2n}(w)}$ and the Kauffman polynomial of $\overline{q_{2n}(w)}$ from this.
\end{corollary}

\begin{proof}

This is the generalization of Theorem 2.1.
\end{proof}

\section{A way to calculate the writhe of a $n$-bridge knot ($n\geq 2$)}

First, assume that the projection onto the $xy$-plane of a $n$-bridge knot $K$ has a  plat presentaion $p_{2n}(w)$ with $w=\sigma_{k_1}^{\epsilon_1}\sigma_{k_2}^{\epsilon_2}\cdot\cdot\cdot
\sigma_{k_{m-1}}^{\epsilon_{m-1}}\sigma_{k_{m}}^{\epsilon_{m}}$ for some non-zero integers $\epsilon_i$ ($1\leq i\leq m$), where $k_i\in\{1,2,\cdot\cdot
\cdot, 2n-1\}$ and $k_j\neq k_{j+1}$ for $1\leq j\leq m-1$.\\

Then we have the  plat presentation $q_{2n}(w)$ of the tangle $T=K\cap B_2$ so that $\overline{q_{2n}(w)}=p_{2n}(w)$.\\ 

Let  $\mathcal{P}(\sigma_i^{\pm 1})$ be the $2n\times 2n$ matrix which is obtained by interchanging the $i$th and $i+1$st rows of $I$.
Then $\mathcal{P}$ extends to a homomorphism from $\mathbb{B}_{2n}$ to $GL_6(\mathbb{Z})$.\\

For an element $w$ of $\mathbb{B}_{2n}$, let 1,2,...,$2n$ be the upper endpoints of the $2n$ strings numbered from the left. Then let $\delta_i$ be the trivial arc components of $p_{2n}(w)\cap p(B_1)$  so that $\partial \delta_i=\{2i-1,2i\}$. Also, let $\gamma_j$ be the trivial arcs which connect pairs of consecutive strings of the braid at the bottom endpoints numbered from the left.
 Let $u=[1,2,\cdot\cdot
\cdot,2n]$. Then we assign the same number to the bottom endpoint of the $2n$ strings.  Then we say that the new ordered sequence of numbers $w(u)$ is the $\emph{permutation  induced by $w$}$.

\begin{lemma}
$w(u)=[1,2,\cdot\cdot\cdot,2n]\mathcal{P}(w)$.
\end{lemma}

\begin{proof}
This is proven by induction on $l=|\epsilon_1|+|\epsilon_2|+\cdot\cdot\cdot+|\epsilon_m|$.

 \end{proof}
 
Let $R$ be the $2n\times 2n$ matrix which is obtained by interchanging  the $2i+1$st and $2i+2$nd rows of $I$ for all $i$ such that $0\leq i\leq n-1$.\\
 
 Recall that $w^r$ is the reverse word of $w$.\\
 
Let $[o_j(1),o_j(2),\cdot\cdot\cdot,o_j(2n) ]=[1,2,\cdot\cdot\cdot,2n](\mathcal{P}(w)R\mathcal{P}(w^{r})R)^{j-1}\mathcal{P}(w)R\mathcal{P}(w^{r})$ for $1\leq j\leq n$.\\

 Also, let $[o_j'(1),o_j'(2),\cdot\cdot\cdot,o_j'(2n) ]=[1,2,\cdot\cdot\cdot,2n](\mathcal{P}(w)R\mathcal{P}(w^{r})R)^j$ for $0\leq j\leq n$.\\

  We note that
$[o_0'(1),o_0'(2),\cdot\cdot\cdot,o_0'(2n) ]=u$.\\

 Consider the case that $j=1$. In order to get $[1,2,\cdot\cdot\cdot,2n]\mathcal{P}(w)$,  we follow the strings of the braid down while preserving the numbers of the strings. Then  $[1,2,\cdot\cdot\cdot,2n]\mathcal{P}(w)R$ is obtained by following along the trivial arcs $\gamma_i$  while preserving the numbers of the strings.
Then we follow the strings of the braid up while preserving the numbers of the strings to get $[1,2,\cdot\cdot\cdot,2n]\mathcal{P}(w)R\mathcal{P}(w^r)$ which is $[o_1(1),o_1(2),\cdot\cdot\cdot,o_1(2n) ]$. After this, we follow along the trivial arcs $\delta_i$  while preserving the numbers of the strings to get $[1,2,\cdot\cdot\cdot,2n]\mathcal{P}(w)R\mathcal{P}(w^r)R$ which is
$[o_1'(1),o_1'(2),\cdot\cdot\cdot,o_1'(2n) ]$.\\

Generally speaking, from the $[o_j'(1),o_j'(2),\cdot\cdot\cdot,o_j'(2n) ]$ we follow the strings of the braid  down and follow along the trivial arcs $\gamma_i$  and follow the strings of the braid up to get $[o_{j+1}(1),o_{j+1}(2),\cdot\cdot\cdot,o_{j+1}(2n) ]$ while preserving the numbers of the strings.
Then, we get $[o'_{j+1}(1),o'_{j+1}(2),\cdot\cdot\cdot,o'_{j+1}(2n) ]$ by following along the trivial arcs $\delta_i$  while preserving the numbers of the strings.\\

We note that $o'_0(i)=i, o'_1(i),\cdot\cdot\cdot o'_{n-1}(i)$ are distinct points. Otherwise, $K$ is a link, not a knot. Also, we know that $o'_j(i)=o'_{n+j}(i)$.\\

Similarly, $o_1(i), o_2(i),\cdot\cdot\cdot o_{n}(i)$ are distinct points and $o_j(i)=o_{n+j}(i)$.\\

Also, we note that for a trivial arc $\delta_k$ there exists a unique $i$ ($1\leq i\leq n$) so that either  $o_i(1)=2k-1$ or $o_i(1)=2k$. \\

Without loss of generality, give the clockwise orientation  to the trivial arc $\delta_1$ in $B_1$ with $\partial\delta_1=\{1,2\}$ from $1$ to $2$ along $\delta_1$. So,  the initial point of $\delta_1$ is 1 and the terminal point of $\delta_1$ is 2 for the given orientation. Then, the orientations of the other trivial arcs $\delta_2,\cdot\cdot\cdot,\delta_n$ in $B_1$ are determined by the orientation of the knot $K$ which is induced by $\delta_1$.\\

 \begin{lemma}
  The trivial arc $\delta_k$  has the same clockwise orientation  as $\delta_1$ if  $k=o_i(1)/2$ for some $i$ ($1\leq i\leq n$).  The trivial arc $\delta_k$ has the opposite orientation (counterclockwise) as $\delta_1$ if  $k=(o_i(1)+1)/2$ for some $i$ ($1\leq i\leq n$).
\end{lemma}

\begin{proof}
If $k=o_i(1)/2$ for some $i$  then the endpoints of $\delta_k$ are $o_i(1)-1$ and $o_i(1)$. Also,  the orientation of $\delta_k$ is from $o_i(1)-1$ to $o_i(1)$. Therefore, the $\delta_k$  has the same orientaion as $\delta_1$.\\

If $k=(o_i(1)+1)/2$ for some $i$  then the endpoints of $\delta_k$ are $o_i(1)$ and $o_i(1)+1$. Also,  the orientation of $\delta_k$ is from $o_i(1)+1$ to $o_i(1)$. Therefore, the $\delta_k$  has the opposite orientation as $\delta_1$.
\end{proof}

Recall the ordered sequence of numbers $u=[1,2,\cdot\cdot\cdot,2n]$. Now, we define a new sequence of numbers $r=[r(1),r(2),\cdot\cdot\cdot,r(2n)]$ as follows.
For the  orientation given above, we replace the original numbers of $u=[1,2,...,2n]$ for the initial points of $\delta_i$  by 1 and we replace the original numbers for the terminal points of $\delta_i$  by 2.\\

Now, let  $r_0=r$.\\

 Let $r_i=[r_i(1),r_i(2),r_i(3),\cdot\cdot\cdot,r_i(2n)]=
 r \mathcal{P}(\sigma_{k_1}^{\epsilon_1}\sigma_{k_2}^{\epsilon_2}\cdot\cdot\cdot \sigma_{k_{i-1}}^{\epsilon_{i-1}}\sigma_{k_{i}}^{\epsilon_i})$ for $1\leq i\leq m$.\\

 Let $v(i)=\left\{\begin{array}{cl}
 0 & $if~$ r_{i-1}(k_i)=r_{i-1}(k_{i}+1) \\
1 &  $if~$ r_{i-1}(k_i)\neq r_{i-1}(k_{i}+1)\\
\end{array}\right.$

 Then we calculate the writhe of $K$ as follows.\\
 
 \begin{theorem}
$w(K)=\sum_{i=1}^m(-1)^{v(i)}\epsilon_i$.
 \end{theorem}

 \begin{proof}
 For the $2n$ strings of the braid $w$, we assign the number $r(k)$ to each string with the upper endpoint $k$ for $1\leq k\leq 2n$.\\
 
  Without loss of generality, we give the orientation (clockwise) to  $\delta_1$ from $1$ to $2$ along $\delta_1$. Then the orientation at 1 is up and the orientation at 2 is down.
 Then we know that the orientation at $j$ is up if $r(j)=1$ and it is down if $r(j)=2$.\\

Fix a value $i$ ($1\leq i\leq m$). \\
 
Case 1: Suppose that $r_{i-1}(k_i)=r_{i-1}(k_{i}+1)$.\\
 
 Then the two strings for the ($\sum_{j=1}^{i-1}|\epsilon_j|+1)$-th crossing have the same orientation since $r_{i-1}(k_i)=r_{i-1}(k_{i}+1)$, i.e., the numbers of the two strings for ($\sum_{j=1}^{i-1}|\epsilon_j|+1)$-th crossing are the same. If  $r_{i-1}(k_i)=r_{i-1}(k_{i}+1)=1$  then  the orientations are up and if  $r_{i-1}(k_i)=r_{i-1}(k_{i}+1)=2$  then  the orientations are down.\\
 
   Then we note that the index $e$ of ($\sum_{j=1}^{i-1}|\epsilon_j|+1)$-th crossing is $+1$ if the crossing is positive and the index $e$ of ($\sum_{j=1}^{i-1}|\epsilon_j|+1)$-th crossing is $-1$ if the crossing is negative.\\
 
We note that all the crossings in $\sigma_{k_i}^{\epsilon_i}$ have the same index. \\
 
 Therefore, the contribution of $\sigma_{k_i}^{\epsilon_i}$ to the writhe is $\epsilon_i$\\
 
Since $v(i)=0$, we check that $(-1)^{v(i)}\epsilon_i=(-1)^0\epsilon_i=\epsilon_i$ is the contribution of $\sigma_{k_i}^{\epsilon_i}$ to the writhe.\\
 
Case 2: Suppose that $r_{i-1}(k_i)\neq r_{i-1}(k_{i}+1)$ .\\

Then we know that the orientations of the two strings for ($\sum_{j=1}^{i-1}|\epsilon_j|+1)$-th crossing are either up and down or down and up since $r_{i-1}(k_i)\neq r_{i-1}(k_{i}+1)$.\\

So, we check that the index $e$ of ($\sum_{j=1}^{i-1}|\epsilon_j|+1)$-th crossing is $-1$ if the crossing is positive and the index $e$ of ($\sum_{j=1}^{i-1}|\epsilon_j|+1)$-th crossing  is $+1$ if the crossing is negative.\\

Therefore, the contribution of $\sigma_{k_i}^{\epsilon_i}$ to the writhe is $-\epsilon_i$\\

Since $v(i)=1$, we  check that $(-1)^{v(i)}\epsilon_i=(-1)\epsilon_i=-\epsilon_i$ is the contribution of $\sigma_{k_i}^{\epsilon_i}$ to the writhe.\\

By adding all the indices of  $\sigma_{k_i}^{\epsilon_i}$, we have the given formula for the writhe.
\end{proof}

\vskip 30pt
Department of Mathematics, Oklahoma State University, 401 Mathematical Sciences, Stillwater, OK 74078, USA

\end{document}